\documentclass[10pt,a4paper]{amsart}
\usepackage{amsfonts, amssymb}
\usepackage[all]{xy}
\usepackage{tikz}
\usepackage[cp1252]{inputenc}
\usepackage{lmodern}
\usepackage[T1]{fontenc}
\usepackage[a4paper]{geometry}
\usepackage[english]{babel}
\usepackage{amsmath, amssymb, amsthm}
\usepackage{stmaryrd}
\usepackage{textcomp}
\usepackage{enumerate}
\usetikzlibrary{shapes.geometric,positioning}

\newtheorem{thm}{Theorem}

\newtheorem{lem}{Lemma}

\newtheorem{defi}{Definition}

\newtheorem{prop}{Proposition}

\newcommand{\Mod}[1]{\ (\textup{mod}\ #1)}
\newcommand{\Zn}{\mathbb{Z}/n\mathbb{Z}}

\begin{document}
\title{On modular $k$-free sets}
\begin{abstract}
Let $n$ and $k$ be integers. A set $A\subset\Zn$ is $k$-free if for all $x$ in $A$, $kx\notin A$. We determine the maximal cardinality of such a set when $k$ and $n$ are coprime. We also study several particular cases and we propose an efficient algorithm for solving the general case. We finally give the asymptotic behaviour of the minimal size of a $k$-free set in $\llbracket 1,n\rrbracket$ which is maximal for inclusion.
\end{abstract}

\author{Victor Lambert}

\maketitle

\section{Introduction}\label{sec1}

Let $k\geqslant1$ be an integer. A set $A\subset\mathbb{N}$ is said to be $k$-free if $x\neq ky$ for all $x,y$ in $A$. Wang first investigated in 1989 the problem of $2$-free sets in the integers and, using elementary tools, he proved in \cite{W} that the maximal density of a $2$-free set in $\llbracket1,n\rrbracket:=\left\{1,\ldots,n\right\}$ is $2/3$. More recently, Wakeham and Wood studied in \cite{WW} a generalisation of $2$-free sets into $\left\{a,b\right\}$-multiplicative sets ($ax\neq by$ for all $x,y \in A$). Notice that $k$-free sets are the particular case of $\left\{1,k\right\}$-multiplicative sets. They studied this problem through graph theory to get the maximal size of such a set. In particular, they showed that the maximal density of a $k$-free set in $\llbracket1,n\rrbracket$ is $k/(k+1)$. 

Beyond their own interest, $k$-free sets are useful for the study of $k$-fold Sidon sets. Those sets were first introduced by Lazebnik and Verstra\"{e}te in \cite{LV} through a work on the generalize Tur\'{a}n number. 
\begin{defi}A set $A\subset \mathbb{Z}$ is a $k$-fold Sidon set if $A$ has only trivial solutions to each equation of the form $c_1x_1+c_2x_2+c_3x_3+c_4x_4=0$ where $0\leqslant \left|c_i\right|\leqslant k$, and $c_1+c_2+c_3+c_4=0$.\end{defi} \noindent A $1$-fold Sidon set is a Sidon set in the usual sense ($x_1+x_2=x_3+x_4$ has only trivial solutions). If we denote by $D^*(A)=\left\{a_1-a_2, a_1\neq a_2\in A\right\}$ the set of differences from $A$, without $0$, a $2$-fold Sidon set $A$ is a Sidon set which has also the property that $D^*(A)$ is a $2$-free set. More generally, for a $k$-fold Sidon set $A$, $D^*(A)$ is a $k'$-free set, for each $k'\leqslant k$. Using only this fact, Cilleruelo and Timmons proved in \cite{Ci} that for any integer $k\geqslant 1$, a $k$-fold Sidon set $A\subset\llbracket0,n\rrbracket$ has at most $\left(n/k\right)^{1/2}+O((nk)^{1/4})$ elements. 

We only know that the main term $\left(n/k\right)^{1/2}$ is optimal for $k=1$. Indeed, Sidon sets have been widely studied (see \cite{O} for a survey) and there exist three constructions of maximal Sidon sets in $\Zn$ for some $n$. Bose and Chowla proved in \cite{BC} the existence of a Sidon set of size $q+1$ in $\mathbb{Z}/\left(q^2+q+1\right)\mathbb{Z}$ (Singer's sets, see also \cite{S}) and $q$ in $\mathbb{Z}/\left(q^2-1\right)\mathbb{Z}$ (Bose's sets) where $q$ is a power of a prime. Rusza also made an optimal construction in \cite{R} for $\mathbb{Z}/\left(p^2-p\right)\mathbb{Z}$ where $p$ is a prime number. For $k=2$, if $n=2^{2^t+1}+2^t+1$ with $t$ a positive integer, we can extract (see \cite{LV}) from a Singer's set a $2$-fold Sidon set in $\Zn$ of size $$\left|A\right|\geqslant \frac{n^{1/2}}{2}-3.$$
For $k\geqslant3$, we do not even know if there exists a constant $c_k>0$ such that for all integers $n\geqslant 1$, there is a $k$-fold Sidon set $A \subset \llbracket0,n\rrbracket$ with $\left|A\right|\geqslant c_kn^{1/2}$.

In all these problems, we see that it is important and useful to study the case of modular sets. In this paper, we will study $k$-free sets in $\Zn$. Notice that we cover the case of $\left\{a,b\right\}$-multiplicative set in $\Zn$ for some $a, b$ and $n$. Indeed, if $\gcd(a,n)=1$, an $\left\{a,b\right\}$-multiplicative set in $\Zn$ is a $ba^{-1}$-free set. 

We denote $$R_k(n)=\max\left\{\left|A\right|, A \text{ is a } k\text{-free set in }\Zn\right\}$$ and we show in this article how to compute this quantity recursively in $n$ (Theorems \ref{kprimen}, \ref{km}, \ref{k2m} and \ref{algo}). Proofs also give a way to construct a $k$-free set of maximal size.

The study of this quantity strongly depends on the arithmetical relative properties of $n$ and $k$, that is why we split the results in four theorems. We first deal with the case where $k$ and $n$ are coprime, which is actually the most important case. Indeed, when we define $k$-fold Sidon sets in $\Zn$, we must add the condition that $n$ is relatively prime to all integers in $\llbracket1,k\rrbracket$. Otherwise, one could have $c_i(a_1-a_2)=0$ with $a_1\neq a_2$ for some $\left|c_i\right|\leqslant k$, which leads to a nontrivial solution to $c_i(x_1-x_2)+x_3-x_4=0$ for example. 

For $k$ and $d$ integers, we denote by $l_k(d)$ the multiplicative order of $k$ in $\left(\mathbb{Z}/d\mathbb{Z}\right)^*$. We also use the notations $I$ for the indicator function of odd numbers and $\varphi$ for the Euler indicator function. Let see now with the first result below how to compute $R_k(n)$ in the case $\gcd(n,k)=1$.

\begin{thm}\label{kprimen}
If $\gcd(n,k)=1$, $$R_k(n)= \frac{n-1}{2}-\sum_{d|n, d\neq1}{\frac{\varphi(d)I(l_k(d))}{2l_k(d)}}.$$
\end{thm}

For the problem of upper bound for the size of a $2$-fold Sidon set, we are interested in small $R_2(n)$. Indeed, if $n=2^m-1$ is a Mersenne prime number, which implies $m$ prime, then $l_2(n)=m$, hence $$R_2(n)=\frac{n-1}{2}-\frac{n-1}{2\log_2(n-1)}$$
which leads to an upper bound for the size of a $2$-fold Sidon set $A$ :
$$\left|A\right|\leqslant\sqrt{\frac{n-1}{2}-\frac{n-1}{\log_2(n-1)}+\frac{1}{4}}+\frac{1}{2}.$$
Moreover, we prove in Section \ref{sec3} that for fixed $k$ the error term is $o(n)$. Thus $R_k(n)=(n-1)/2-o(n)$.

When $k$ divides $n$, the problem becomes easier and we have the two following results.

\begin{thm}\label{km}
If $m$ is not divisible by $k$, then $$R_k(km)=(k-1)m.$$
\end{thm}
When $k^2$ divides $n$, we get a recursive formula. That is the purpose of Theorem \ref{k2m}. 
\begin{thm}\label{k2m}
Let $k, m$, and $n$ be integers. Then, we have :
$$R_k(k^2m)=R_k(m)+\left(k^2-k\right)m.$$
\end{thm}

Notice that Theorems \ref{kprimen}, \ref{km} and \ref{k2m} cover all cases when $k$ is prime. Moreover, recall that the maximal density of a $k$-free set in $\llbracket1,n\rrbracket$ is $k/(k+1)$. In the modular case, applying Theorem \ref{k2m} we get $$R_k(k^{2m})=\frac{k}{k+1}\left(k^{2m}-1\right)$$ which lead to the next proposition.

\begin{prop}
Let $k$ be an integer, $k\geqslant 1$, we have $$\limsup_n\frac{R_k(n)}{n}=\frac{k}{k+1}.$$
\end{prop}

Now, to illustrate the two last theorems, let consider an example. We compute $R_{15}(826875)$ :

$$\begin{aligned}
R_{15}(826875)&=R_{3.5}(3^3.5^4.7^2)\\
&=R_{3.5}(3.5^2.7^2)+(15^2-15).3.5^2.7^2\\
&=(15-1)5.7^2+(15^2-15).3.5^2.7^2\\
&=775180.
\end{aligned}$$
We will consider again this example in Section \ref{sec5}.

In the general case, we cannot obtain a closed formula, but in Sections \ref{sec2} and \ref{sec4} we propose an efficient algorithm to compute $R_k(n)$.

\begin{thm}\label{algo}
There exists an algorithm which provides the maximal size of a $k$-free set in $\Zn$ and a method to construct one in $O((\log(n))^2)$ operations.
\end{thm}

To get this complexity, we must assume that we know the prime factorization of $k$ and $n$, which is unfortunately hard to obtain in general. However, we can easily apply the algorithm to compute our function $R_k$ for new types of $k$ and $n$. That is the purpose of the theorem below.

\begin{thm}\label{cas}
Let $p$ and $q$ be prime numbers, $\alpha$, $\beta$ and $u$ be integers.
\begin{enumerate}
\item If $\gcd(u,p)=1$, $$R_{up}(p^\alpha)=\sum_{i=0}^{\left\lfloor \frac{\alpha-1}{2}\right\rfloor}\varphi(p^{\alpha-2i}).$$
\item If $\gcd(u,p)=1$, $$R_{up^2}(p^\alpha)=\sum_{i=0}^{\left\lfloor \frac{\alpha-1}{4}\right\rfloor}\left(\varphi(p^{\alpha-4i})+\varphi(p^{\alpha-4i-1})\right).$$
\item If $\gcd(u,p)=\gcd(u,q)=1$, $$R_{up}(p^\alpha q^\beta)=\sum_{j=0}^\beta\sum_{i=0}^{\left\lfloor \frac{\alpha-1}{2}\right\rfloor}\varphi(p^{\alpha-2i}q^{\beta-j})$$
$$R_{up^2}(p^\alpha q^\beta)=\sum_{j=0}^{\beta}\sum_{i=0}^{\left\lfloor \frac{\alpha-1}{4}\right\rfloor}\left(\varphi(p^{\alpha-4i}q^{\beta-j})+\varphi(p^{\alpha-4i-1}q^{\beta-j})\right).$$
\end{enumerate}
\end{thm}

In the same way, we could obviously go further and study the case $k=up^3$ or $n=p^\alpha q^\beta r^\gamma$ for instance, but that would give very unpleasant formulas.

Next, we study $k$-free sets in the set of integers, and not in modular sets anymore. We wonder what is the minimal size of a $k$-free set in $\llbracket1,n\rrbracket$ which is maximal for inclusion, and we answer it in the following theorem, where we define $$\tilde{R_k}(n)=\min\left\{\left|A\right|, A \subset \llbracket1,n\rrbracket \text{ a k-free set which is maximal for inclusion}\right\}.$$

\begin{thm}\label{th2}
$$\tilde{R_k}(n)=\frac{k^2}{k^2+k+1}n+O(\log_k^2(n)).$$
\end{thm}

In the next section, we introduce some notations and give three lemmas. Section \ref{sec3} contains the proof of Theorems \ref{kprimen}, \ref{km} and \ref{k2m}. In Section \ref{sec4}, we study and prove the algorithm for the general case, which we use in Section \ref{sec5}. We conclude by the proof of Theorem \ref{th2} in the last section.

\section{Preparatory lemmas}\label{sec2}

Let introduce some useful notations for our study. We define $\mathcal{O}^k(x):=\left\{k^jx, j\in \mathbb{N}\right\}$ and we call it the orbit of $x$ (by the multiplication by $k$). We use it in a different context (in $\Zn$ or in $\mathbb{N}$, Section \ref{sec7}) with the same notation. We denote by $k\cdot A:=\left\{ka, a\in A\right\}$ the dilated set of $A$ and by $A_m$ the subset of $\Zn$ $$A_m:=\left\{x, \gcd(x,n)=m\right\}=\left\{x=mu, \gcd\left(u,\frac{n}{m}\right)=1\right\}$$
and we have $\left|A_m\right|=\varphi(n/m)$.

To study $k$-free sets, it is important to know more about $\mathcal{O}^k(x)$ for each $x\in \Zn$. That is the purpose of our first two lemmas.

\begin{lem}\label{kAm}
If we write $$n=\prod_{i=1}^rp_i^{n_i}\prod_{i=r+1}^sp_i^{n_i} \text{ and } k=u\prod_{i=1}^rp_i^{k_i}$$ with $\gcd(u,p_i)=1, \forall i \in \llbracket1,s\rrbracket$, then for $m$ which divides $n$, m has the form $$m=\prod_{i=1}^rp_i^{m_i}\prod_{i=r+1}^sp_i^{m_i}$$ with $m_i\leqslant n_i, \forall i \in \llbracket1,s\rrbracket$, then we have 
$$k\cdot A_m=A_{m^\prime} \text{ where } m^\prime=m\prod_{i=1}^rp_i^{\min(k_i,n_i-m_i)}.$$
\end{lem}

\begin{proof}
Let $x\in A_m$, then $x=mv$ with $\gcd(v,n/m)=1$. Thus, we get

$$\begin{aligned}
\gcd(kx,n)&=m\gcd(kv,\frac{n}{m})\\
&=m\gcd(k,\frac{n}{m})\\
&=m\gcd(\gcd(k,n),\frac{n}{m})\\
&=m\gcd(\prod_{i=1}^rp_i^{k_i},\prod_{i=1}^rp_i^{n_i-m_i}\prod_{i=r+1}^sp_i^{n_i-m_i}).
\end{aligned}
$$
What we get is that $k\cdot A_m\subset A_{m^\prime}$.

Conversely, there exists now $y\in A_{m^\prime}$ such that $y=kx$ with $x\in A_m$ (since $A_m\neq \emptyset$). But for all $z$ in $A_{m^\prime}$, there exists $w$, $\gcd(w,n)=1$ and $z=wy$. Clearly, $xw$ belongs to $A_m$ and $z=kxw$, which concludes the proof.

\end{proof}

\begin{lem}\label{lem2}
Let $m$ be a divisor of $n$, $k$ be an integer such that $\gcd(k,n/m)=1$ and $x\in A_m$. Then $$\left|\mathcal{O}^k(x)\right|=l_k\left(\frac{n}{m}\right).$$

\end{lem}

\begin{proof}
Since $x\in A_m$, if we denote by $<x>$ the subgroup generated by $x$, we have $$<x>\cong \mathbb{Z}/\left(\frac{n}{m}\right)\mathbb{Z}.$$
Then, since $k$ is invertible in this subgroup $$\mathcal{O}^k(x)\cong \mathcal{O}^k(1)=<k>\subset \mathbb{Z}/\left(\frac{n}{m}\right)\mathbb{Z}$$
and the size of $<k>$ in this subgroup is exactly $l_k(n/m)$. 

\end{proof}

Now, we need a result about specific rooted trees. That is the purpose of the next lemma which will be very useful for the proof of Theorem \ref{algo}.

Let $T$ be a rooted tree where the set of nodes is $V=\left\{v_i\right\}_{i\in I}$ with $I$ a finite set and $E$ is the set of edges. We associate a value $\alpha_i\geqslant0$ to each $v_i$ and we denote by $l_i$ its level (recall that the level of a node is defined by $1 +$ the minimal number of connections between the node and the root). Assume that $T$ has the following property :

\begin{equation}\label{tree1}
 \text{If } v_i \text{ is the parent of } v_j, \text{ then } \alpha_i<\alpha_j.
\end{equation}
In other words, $\alpha$ is strictly increasing in each branche. Notice that this condition implies that if $v_i$ is not the root of $T$, $\alpha_i>0$. We search a subset $A$ of $I$ satisfying :

\begin{equation}\label{tree2}
\forall (i,j) \in A^2, (v_i,v_j)\notin E \text{ and } \alpha_i\neq0
\end{equation} 
which maximizes the quantity $$\Lambda_A=\sum_{i\in A}\alpha_i.$$

We denote by $l$ the maximal level in $T$ and we construct a set $B$ by the following algorithm :
Initialization : $B=\left\{v_i | l_i=l\right\}$.
Then for $k$ from $1$ to $l-1$ : for all $i$ such that $l_i=l-k$, we add $v_i$ to $B$ if and only if $\alpha_i\neq 0$ and there is no child of $v_i$ in $B$.

It is clear that $B$ satisfies (\ref{tree2}). Actually, $B$ is the required set for our problem.

\begin{lem}
$B$ is the unique subset of $I$ maximizing $\Lambda_A$ among the sets $A$ satisfying (\ref{tree2}).
\end{lem}

\begin{proof}
We proceed by induction on the size of $I$. If $\left|I\right|=1$ there is nothing to say.

Let $n$ be an integer and assume that the lemma holds for all $k$ less than $n$. Let $\left|I\right|=n+1$, $B$ the set from the algorithm applied to $T$ and $C$ be a subset of $I$ maximizing $\Lambda_A$ among the sets $A$ satisfying (\ref{tree2}). We denote by $v_0$ the root of $T$, and $v_i$, $i \in \llbracket1,K\rrbracket$ the childs of $v_0$. We also define $T_i$ the rooted subtree of $T$ of root $v_i$ for all $i$ in $\llbracket1,K\rrbracket$, $B_i=B\cap T_i$ and $C_i=C\cap T_i$. By induction, for all $i$, $\Lambda_{C_i}\leqslant\Lambda_{B_i}$ with equality if and only if $B_i=C_i$.

If $v_0\in B$ and $v_0\in C$, we have $$\Lambda_{C}-\alpha_0=\sum_{i=1}^K\Lambda_{C_i}\leqslant\sum_{i=1}^K\Lambda_{B_i}=\Lambda_{B}-\alpha_0.$$
Thus, by the definition of $C$, this is an equality, and finally $B=C$.

If $v_0\notin B$ and $v_0\notin C$, we have
$$\Lambda_{C}=\sum_{i=1}^K\Lambda_{C_i}\leqslant\sum_{i=1}^K\Lambda_{B_i}=\Lambda_{B}$$
which ensure that $B=C$ for the same reason.

If $v_0\in B$ and $v_0\notin C$, since $\alpha_0>0$ (otherwise, $v_0$ is not in $B$ according to the algorithm), we have 
$$\Lambda_{C}=\sum_{i=1}^K\Lambda_{C_i}\leqslant\sum_{i=1}^K\Lambda_{B_i}=\Lambda_{B}-\alpha_0<\Lambda_B$$
which leads to a contradiction.

If $v_0\notin B$ and $v_0\in C$, $\alpha_0>0$ (since $C$ satisfies (\ref{tree2})) and it means that there exists $i_0$ in $\llbracket1,K\rrbracket$ such that $v_{i_0}\in B$. Then, we consider the branche from $v_0$ which contains $v_{i_0}$. If $K>1$, its size is strictly less than $n+1$ and we can apply the induction hypothesis to get $\Lambda_{C_{i_0}}+\alpha_0<\Lambda_{B_{i_0}}$. Thus, 

$$\Lambda_{C}=\sum_{i\neq i_0}\Lambda_{C_i}+\Lambda_{C_{i_0}}+\alpha_0<\sum_{i\neq i_0}\Lambda_{B_i}+\Lambda_{B_{i_0}}=\Lambda_{B}$$
and we have a contradiction. If $K=1$, we denote by $v_1$ the only child of $v_0$, $v_1\in B$ and $v_1\notin C$, and considering now the subtrees whose the roots are every nodes of level $2$, we get

$$\Lambda_{C}=\sum\Lambda_{C^\prime_i}+\alpha_0<\sum\Lambda_{B^\prime_i}+\alpha_1=\Lambda_{B}$$
where we use $\alpha_1>\alpha_0$. This is a contradiction.

Finally, $B=C$ in every cases and the lemma is proved.

\end{proof}

\section{Proof of Theorems \ref{kprimen}, \ref{km} and \ref{k2m}}\label{sec3}

We first deal with the Theorem \ref{kprimen}, the case $\gcd(n,k)=1$.

\begin{proof}
As mentioned in the introduction, let $l_k(d)$ be the order of $k$ in $\left(\mathbb{Z}/d\mathbb{Z}\right)^*$ and $I$ be the indicator function of odd numbers.

By lemma \ref{kAm}, $\mathcal{O}^{k}(x)\subset A_m$, for all $x$ in $A_m$. Therefore, we consider the suitable partition 

$$\begin{aligned}
\left(\Zn\right)\setminus\left\{0\right\}&=\bigsqcup_{m|n, m<n}A_m.
\end{aligned}$$
Notice that this partition is trivial if $n$ is prime. By lemma \ref{lem2}, if $x\in A_m$, we have $$\left|\mathcal{O}^{k}(x)\right|=l_k\left(\frac{n}{m}\right).$$
Hence, we can make a partition of $A_m$ in $\varphi(n/m)/l_k(n/m)$ distinct orbits of length $l_k(n/m)$.
In each orbit, to get an optimal $k$-free set, we have to take the most possible elements without taking two consecutive elements. But the orbits are cyclic, that's why if the length $l$ of an orbit is even, we can take $l/2$ elements, whereas if $l$ is odd, we can take only $(l-1)/2$ elements. We finally get the formula

$$\begin{aligned} R_k(n)=& \sum_{d|n, d\neq1}{\frac{\varphi(d)}{l_k(d)}\left(\frac{l_k(d)-I(l_k(d))}{2}\right)}\\
&= \frac{n-1}{2}-\sum_{d|n, d\neq1}{\frac{\varphi(d)I(l_k(d))}{2l_k(d)}}.
\end{aligned}$$

\end{proof}

Actually, if we fix $k$, $R_k(n)$ is asymptotically $(n-1)/2-o(n)$. Indeed, for all $\varepsilon>0$, there exists $d_0$ such that $\log_k{d_0}\geqslant1/\varepsilon$ and there exists $n$ such that $d_0^2/6\leqslant\varepsilon n/2$. Thus,

$$\begin{aligned}
\sum_{d|n, d\neq1}{\frac{\varphi(d)I(l_d)}{2l_d}}&=\sum_{d|n, d\neq1, d\leqslant d_0}{\frac{\varphi(d)I(l_d)}{2l_d}}+\sum_{d|n, d\neq1, d>d_0}{\frac{\varphi(d)I(l_d)}{2l_d}}\\
&\leqslant \sum_{d|n, d\neq1, d\leqslant d_0}{\frac{\varphi(d)}{6}}+\sum_{d|n, d\neq1, d>d_0}{\frac{\varphi(d)}{2log_2{d}}}\\
&\leqslant \frac{d_0^2}{6}+\frac{\varepsilon n}{2}\\
&\leqslant \varepsilon n.
\end{aligned}$$

Now, we consider the case $n=k^2m$, for which we have a suitable partition of $\Zn$ :

\begin{lem}
In this case, we have
$$\Zn=\left(k^2\Zn\right)\bigsqcup\left(\bigcup_{h\not\equiv0\Mod k}\left\{h,kh\right\}\right).$$
\end{lem}

\begin{proof}
Indeed, if $ x \not\equiv0\Mod {k^2}$ and $x \equiv 0\Mod k$, then $ x=kh$ with $h\not\equiv0\Mod k$. Thus, we have all the elements in this union. Moreover, if we have $h\not\equiv0\Mod k$, then $kh\not\equiv0\Mod {k^2}$, which shows that the first union is disjoint.
\end{proof}

Let see now why this is a good repartition of elements for our problem, through the proof of Theorem \ref{k2m} :

\begin{proof}
We remark two main things :

\begin{itemize}
\item $k^2\Zn$ is stable for multiplication by $k$.
\item If $h\not\equiv0\Mod k$, we can not write $h=ku$ in $k^2\Zn$.
\end{itemize}

We consider now $A$ a $k$-free set in $\Zn$. First, for each $h\not\equiv0\Mod k$, at most one of $\left\{h,kh\right\}$ lies in $A$. Furthermore, by the first remark, $A\cap k^2\Zn$ is also a $k$-free set, which can be easily seen equivalent to a $k$-free set in $\mathbb{Z}/m\mathbb{Z}$. This leads to

$$R_k(k^2m)\leqslant R_k(m)+\left|\left\{h\not\equiv0\Mod k\right\}\right|=R_k(m)+\left(k^2-k\right)m.$$

Let see now the construction of an optimal $k$-free set. By the second remark, we can take every $h\not\equiv0\Mod k$ in $A$, and we now that $kh\notin k^2\Zn$, so we can take $R_k(m)$ elements from $k^2\Zn$ in $A$. Thus, we get

$$R_k(k^2m)=R_k(m)+\left(k^2-k\right)m$$
and that concludes the proof.

\end{proof}

Finally, we consider $n=km$ with $m\not\equiv0\Mod k$. In this case, we have :

\begin{lem}
$$\Zn=\bigcup_{h\not\equiv0\Mod k}\left\{h,kh\right\}.$$
\end{lem}

\begin{proof}
If $ x\equiv0\Mod k$, there exists $u$ such that $x=ku$. If $u\not\equiv0\Mod k$, $x$ is in the right form. Else, $u\equiv0\Mod k$, then there exists $v$, $u=kv$ and we have $x=x+n=x+km=k^2v+km$. But $m\not\equiv0\Mod k$ by hypothesis, then we can write $m=lk+a$ with $a\not\equiv0\Mod k$. We get $$x+km=k(kv+lk+a).$$
Since $h=kv+lk+a\not\equiv0\Mod k$, we have written $x=x+km=kh$ with $h\not\equiv0\Mod k$, which concludes the lemma.
\end{proof}

We can now easily prove Theorem \ref{km}.

\begin{proof}
If $A$ is a $k$-free set, for each $h\not\equiv0\Mod k$, at most one of $\left\{h,kh\right\}$ lies in $A$, then $\left|A\right|\leqslant (k-1)m$. If $h\not\equiv0\Mod k$, we can not write $h=ku$ in $\Zn$ since $n=km$. Thus $\left\{h\not\equiv0\Mod k\right\}$ is a $k$-free set and we get

$$R_k(km)=(k-1)m.$$

\end{proof}

\section{Theorem \ref{algo} : the case $\gcd(k,n)\neq1$}\label{sec4}

The aim is to define a forest (a disjoint union of rooted trees) satisfying (\ref{tree1}) such that the algorithm of Section \ref{sec2} gives an optimal $k$-free set. Recall that if $m$ divides $n$, we denote by $A_m$ the subset of $\Zn$ $$A_m:=\left\{x, \gcd(x,n)=m\right\}$$
and we have $\left|A_m\right|=\varphi(n/m)$. The disjoint union of $A_m$ is a partition of $\Zn$.

Let $G=(V,E)$ be a graph with the set of vertices $V=\left\{m\right\}_{m|n}$ and we define the set of edges $E$ by :
\begin{equation}\label{G}
(m,m^\prime)\in E \text{ if and only if } m<m^\prime \text{ and } k\cdot A_m=A_{m^\prime}.
\end{equation}
We first need to well understand this graph, then we will associate suitable values to vertices for our problem.
We write $$n=\prod_{i=1}^rp_i^{n_i}\prod_{i=r+1}^sp_i^{n_i} \text{ and } k=u\prod_{i=1}^rp_i^{k_i}$$ with $\gcd(u,p_i)=1, \forall i \in \llbracket1,s\rrbracket$ and $k_i>0$ for all $i$. Let denote by $\mathcal{M}$ the set of divisors of $n$ of the form $$m=\prod_{i=1}^sp_i^{m_i}$$
with $m_i\leqslant n_i,  \forall i\in \llbracket 1,s\rrbracket$, such that there exists $i_0\leqslant r$ satisfying $m_{i_0}<\min(k_{i_0},n_{i_0})$. The next proposition gives the structure of $G$.

\begin{prop}
$G$ is a disjoint union of rooted trees. Furthermore :
\begin{enumerate}[(i)]
\item A connected component is defined by the choice of $\left\{d_i\right\}_{i=r+1\ldots s}$ with $d_i\leqslant n_i$.
\item The leaves are exactly the elements of $\mathcal{M}$.
\item The root of the tree defined by $\left\{d_i\right\}_{i=r+1\ldots s}$ is $$m=\prod_{i=1}^rp_i^{n_i}\prod_{i=r+1}^sp_i^{d_i}.$$
\item The level of $m$ is $j_m+1$ where $$j_m=\min\left\{j | jk_i\geqslant n_i-m_i , \forall i\in \llbracket1,r\rrbracket\right\}.$$ 

\end{enumerate}
\end{prop}

\begin{proof}
We define 
$$k^j\ast m=m\prod_{i=1}^r p_i^{\min(jk_i,n_i-m_i)}.$$
By lemma \ref{kAm}, $A_{k\ast m}=k\cdot A_m$, then if $(m,m^\prime)$ is an edge, we have $m_i=m^\prime_i$ for all $i$ in $\llbracket r+1, s\rrbracket$. Thus, if there exists a path between two vertices, they have the same $\left\{d_i\right\}_{i=r+1\ldots s}$.

The next lemma shows that a vertice is either in $\mathcal{M}$ or has a descendant in $\mathcal{M}$.

\begin{lem}
Let $m^\prime=\prod_{i=1}^sp_i^{m^\prime_i}$ be a divisor of $n$ which is not in $\mathcal{M}$, then there exists $t>0$ and $m$ in $\mathcal{M}$ such that $m^\prime=k^t\ast m$.
\end{lem}

\begin{proof}
Let $t$ defined by $$t=\min\left\{j | \exists i_0\leqslant r, m^\prime_{i_0}-jk_{i_0}<\min(k_{i_0},n_{i_0})\right\}$$
and define $\alpha_i=\max(0,m^\prime_i-tk_i)$ and $$m=\prod_{i=1}^rp_i^{\alpha_i}\prod_{i=r+1}^{s}p_i^{m^\prime_i}$$ which belongs to $M$ by definition of $t$. Notice that $t>0$ since $m^\prime\notin \mathcal{M}$. Thus, we have
$$\begin{aligned}
k^t\ast m&=m\prod_{i=1}^r p_i^{\min(tk_i,n_i-m_i)}\\
&=\prod_{i=1}^rp_i^{\alpha_i+\min(tk_i,n_i-\alpha_i)}\prod_{i=r+1}^{s}p_i^{m^\prime_i}.
\end{aligned}
$$
We have to study three cases :

\begin{itemize}
\item $\alpha_i=0$ and $k_i<n_i$ : $m^\prime_i\geqslant k_i$ since $m^\prime\in M$, then $m^\prime_i=tk_i$ by the definition of $t$ and we get in this case $\alpha_i+\min(tk_i,n_i-\alpha_i)=m^\prime_i$.
\item $\alpha_i=0$ and $n_i\leqslant k_i$ : $m^\prime_i=n_i$ since $m'\in M$ and  we have $\alpha_i+\min(tk_i,n_i-\alpha_i)=n_i=m^\prime_i$. 
\item Otherwise, $n_i-\alpha_i=n_i-m^\prime_i+tk_i\geqslant tk_i$, then $\alpha_i+\min(tk_i,n_i-\alpha_i)=m^\prime_i$.
\end{itemize}
We finally get $k^t\ast m=m^\prime$, as we expected.
\end{proof}
Conversely, if $m\in M$ and $t>0$, $k^t\ast m \notin M$, and we get that vertices in $\mathcal{M}$ have no child.
Moreover, if we consider $$m=\prod_{i=1}^rp_i^{n_i}\prod_{i=r+1}^sp_i^{d_i}$$ it is clear that $k\ast m=m$, it means that $m$ have no parent. Finally, if $m^\prime$ have the same $\left\{d_i\right\}_{i=r+1\ldots s}$, we have $k^{j_m}\ast m^\prime=m$ and $k^{j_m-1}\ast m^\prime\neq m$ by the definition of $j_m$.

Through those observations, we get the conclusions of the proposition.

\end{proof}

Now, we need to see how to give valuations for vertices. The main problem comes from roots, which are the $m$ satisfying $k\cdot A_m=A_m$. The next lemma computes the maximal size of a $k$-free set in $A_m$ when $m$ is a root of our graph.

\begin{lem}
If $m$ is a root of our graph (which is tantamount to $\gcd(k,n/m)=1$), the maximum size of a $k$-free set included in $A_m$ is
$$R_k(A_m):=\frac{\varphi(n/m)}{l_k(n/m)}\left(\frac{l_k(n/m)-I(l_k(n/m))}{2}\right).$$
\end{lem}

\begin{proof}
We have the isomorphism :
$$A_m\cong A^\prime_1:=\left\{x\in\mathbb{Z}/(n/m)\mathbb{Z}, \gcd(x,\frac{n}{m})=1\right\}.$$
But we are in the case $\gcd(k,n/m)=1$, and if we look through the proof of Theorem \ref{kprimen}, we get immediately the result.

\end{proof}

Thus, we define the valuation of vertices for all $m$ which divides $n$ : 
$$\alpha_m=\left\{\begin{aligned}  &R_k(A_m) \text{ if } m \text{ is a root} \\
&\varphi\left(\frac{n}{m}\right) \text{ otherwise}.
\end{aligned}\right.$$
Notice that our graph has the property \ref{tree1}, which we recall here :
$$ \text{If } v_i \text{ is the parent of } v_j, \text{ then } \alpha_i<\alpha_j.$$
When we apply the algorithm of Section \ref{sec2}, we get a set $B$ of vertices. To construct a $k$-free set, we can join $A_m$ for $m$ in $B$ and not a root, and for the roots $m$ in $B$ we can take $K_m$ a maximal $k$-free set in $A_m$. More precisely, we define 
$$\overline{B}:=\left(\bigsqcup_{\substack{m\in B \\ \gcd(k,n/m)\neq1}}A_m\right)\bigsqcup\left(\bigsqcup_{\substack{m\in B \\ \gcd(k,n/m)=1}} K_m\right)$$
which is clearly a $k$-free set since $B$ satisfies \ref{tree2} and by the definition of $K_m$.

\begin{prop}\label{compute}
$\overline{B}$ is an optimal $k$-free set in $\Zn$ and has size
$$\sum_{\substack{m\in B \\ \gcd(k,n/m)\neq1}}\varphi\left(\frac{n}{m}\right)+\sum_{\substack{m\in B \\ \gcd(k,n/m)=1}}R_k(A_m).$$
\end{prop}

\begin{proof}
Assume that $C$ is a $k$-free set in $\Zn$ with $\left|C\right|>\left|\overline{B}\right|$. Let $x$ be an element in $C\backslash \overline{B}$ of maximal level $t$, $m$ the integer such that $x\in A_m$ and $T_i$ the rooted tree which contains $m$.

First case : $t=1$ and $m\notin B$. Thus, $m$ is a root but not in $B$, which means that there is a child $m^\prime$ of $m$ in $B$ (otherwise $\alpha_m=R_k(A_m)=0$ and $C$ could not be a $k$-free set). Then, the set $k^{-1}(x)=\left\{y \in A_{m^\prime}| y=kx\right\}$ has no element in $C$ but has size $$\left|k^{-1}(x)\right|=\frac{\varphi(n/m^\prime)}{\varphi(n/m)}>1$$ and by substituting $\left\{x\right\}$ by $k^{-1}(x)$, we get a $k$-free set (since $t$ is the maximal level of an element of $C\backslash\overline{B}$) of size strictly greater than $C$.

Second case : $t>1$. By the construction of $\overline{B}$, $m$ does not belong to $B$ and we can do as in the previous case.

The two cases lead to a contradiction, then all elements $x$ in $C\backslash\overline{B}$ satisfy $t=1$ and $m$ belongs to $B$. Thus, $m$ is a root and we can substitute $C\cap A_m$ by $K_m$ for each root, and we get $\left|C\right|\leqslant\left|\overline{B}\right|$.
We finally get the result by counting the size of $\overline{B}$.

\end{proof}

Thus, to get $R_k(n)$, if the prime factorization of $k$ and $n$ is known, we need to construct the graph ($O(\log(n))$ operations), to apply the algorithm ($O(\log(n))$), to compute $\alpha_m$ for $m$ in $B$ ($O((\log(n))^2)$ operations since we have the prime factorization of $m$) and finally add those values.

\section{Applications of Theorem \ref{algo}}\label{sec5}

Now, to illustrate the method in a particular case, we deal with the example mentioned in introduction, which is $n=3^3.5^4.7^2=826875$ and $k=3.5=15$. In this case, we get a forest with roots $3^3.5^4$, $3^3.5^4.7$ and $3^3.5^4.7^2$. We just represent below one of those trees. To get the second, we have to multiply each vertice by $7$, and for the third, by $7^2$. Applying the algorithm, we get :

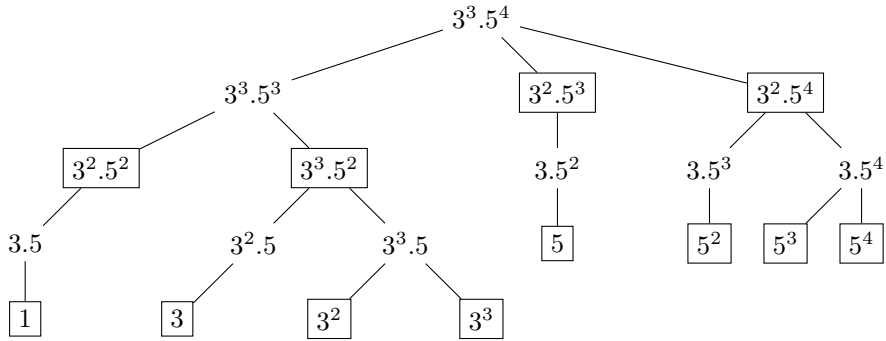
\begin{center}
\begin{tikzpicture}
\node (b0) at (0,0) {$3^3.5^4$};
\node (b11) at (-3,-1) {$3^3.5^3$};
\node[rectangle,draw] (b12) at (1,-1) {$3^2.5^3$};
\node[rectangle,draw] (b13) at (4,-1) {$3^2.5^4$};
\node[rectangle,draw] (b111) at (-5,-2) {$3^2.5^2$};
\node[rectangle,draw] (b112) at (-2,-2) {$3^3.5^2$};
\node (b121) at (1,-2) {$3.5^2$};
\node (b131) at (3,-2) {$3.5^3$};
\node (b132) at (5,-2) {$3.5^4$};
\node (b1111) at (-6,-3) {$3.5$};
\node (b1121) at (-3,-3) {$3^2.5$};
\node (b1122) at (-1,-3) {$3^3.5$};
\node[rectangle,draw] (b1211) at (1,-3) {$5$};
\node[rectangle,draw] (b1311) at (3,-3) {$5^2$};
\node[rectangle,draw] (b1321) at (4,-3) {$5^3$};
\node[rectangle,draw] (b1322) at (5,-3) {$5^4$};
\node[rectangle,draw] (b11111) at (-6,-4) {$1$};
\node[rectangle,draw] (b11211) at (-4,-4) {$3$};
\node[rectangle,draw] (b11221) at (-2,-4) {$3^2$};
\node[rectangle,draw] (b11222) at (0,-4) {$3^3$};

\draw (b0) -- (b11);
\draw (b0) -- (b12);
\draw (b0) -- (b13);
\draw (b11) -- (b111);
\draw (b11) -- (b112);
\draw (b111) -- (b1111);
\draw (b1111) -- (b11111);
\draw (b112) -- (b1121);
\draw (b112) -- (b1122);
\draw (b1121) -- (b11211);
\draw (b1122) -- (b11221);
\draw (b1122) -- (b11222);
\draw (b12) -- (b121);
\draw (b121) -- (b1211);
\draw (b13) -- (b131);
\draw (b13) -- (b132);
\draw (b131) -- (b1311);
\draw (b132) -- (b1321);
\draw (b132) -- (b1322);

\end{tikzpicture}
\end{center}
To get the maximal size of a $15$-free set in $\mathbb{Z}/826875\mathbb{Z}$ we have to sum all $\varphi(n/m)$ for all $m$ choosen by the algorithm in each tree. And we get $R_{15}(826875)=775180$ as we deduced from Theorems \ref{km} and \ref{k2m}.

This way to compute $R_k(n)$ does not give a general formula, that is why we study in Theorem \ref{cas} theorem several cases, which we prove here.

\begin{proof}

\begin{enumerate}
\item The first graph below is the one we get in this particular case ($n=p^\alpha$, $k=up$ with $\gcd(u,p)=1$), then we apply the algorithm and we obtain a set of vertices, which are the one with a box around. We get the second graph when $\alpha$ is even and the third if $\alpha$ is odd :
\begin{center}
\begin{tikzpicture}
\node (b0) at (0,0) {$p^\alpha$};
\node (b1) at (0,-1) {$p^{\alpha-1}$};
\node (b2) at (0,-2) {$\vdots$};
\node (b3) at (0,-3) {$p^2$};
\node (b4) at (0,-4) {$p$};
\node (b5) at (0,-5) {$1$};

\draw (b0) -- (b1);
\draw (b1) -- (b2);
\draw (b2) -- (b3);
\draw (b3) -- (b4);
\draw (b4) -- (b5);

\node (c0) at (3,0) {$p^\alpha$};
\node (c1) at (3,-1) {$p^{\alpha-1}$};
\node (c2) at (3,-2) {$\vdots$};
\node[rectangle,draw] (c3) at (3,-3) {$p^2$};
\node (c4) at (3,-4) {$p$};
\node[rectangle,draw] (c5) at (3,-5) {$1$};

\node (legende2) at (3,-5.5) {$\alpha$ even};

\draw (c0) -- (c1);
\draw (c1) -- (c2);
\draw (c2) -- (c3);
\draw (c3) -- (c4);
\draw (c4) -- (c5);

\node (d0) at (6,0) {$p^\alpha$};
\node[rectangle,draw] (d1) at (6,-1) {$p^{\alpha-1}$};
\node (d2) at (6,-2) {$\vdots$};
\node[rectangle,draw] (d3) at (6,-3) {$p^2$};
\node (d4) at (6,-4) {$p$};
\node[rectangle,draw] (d5) at (6,-5) {$1$};
\node (legende3) at (6,-5.5) {$\alpha$ odd};

\draw (d0) -- (d1);
\draw (d1) -- (d2);
\draw (d2) -- (d3);
\draw (d3) -- (d4);
\draw (d4) -- (d5);

\end{tikzpicture}
\end{center}
Since $A_{p^\alpha}=\left\{0\right\}$, $R_{up}(p^\alpha)=0$, that is why $p^\alpha$ is never considered by the algorithm. By applying the Proposition \ref{compute}, we get the result.

\item We give below the results of the algorithm ($n=p^\alpha$, $k=up^2$ with $\gcd(u,p)=1$), which depends on the value of $\alpha$ modulo $4$ (notice that $R_{up^2}(p^\alpha)=0$) :

\begin{center}
\begin{tikzpicture}
\node (b0) at (0,0) {$p^\alpha$};
\node[rectangle,draw] (b11) at (-0.6,-1) {$p^{\alpha-1}$};
\node (b12) at (0.6,-1) {$p^{\alpha-2}$};
\node (b21) at (-0.6,-2) {$\vdots$};
\node (b22) at (0.6,-2) {$\vdots$};
\node[rectangle,draw] (b31) at (-0.6,-3) {$p^4$};
\node (b32) at (0.6,-3) {$p^3$};
\node (b41) at (-0.6,-4) {$p^2$};
\node[rectangle,draw] (b42) at (0.6,-4) {$p$};
\node[rectangle,draw] (b51) at (-0.6,-5) {$1$};
\node (legende1) at (0,-5.5) {$\alpha\equiv 1\Mod 4$};

\draw (b0) -- (b11);
\draw (b0) -- (b12);
\draw (b11) -- (b21);
\draw (b21) -- (b31);
\draw (b31) -- (b41);
\draw (b41) -- (b51);
\draw (b12) -- (b22);
\draw (b22) -- (b32);
\draw (b32) -- (b42);

\node (a0) at (2.8,0) {$p^\alpha$};
\node (a11) at (2.2,-1) {$p^{\alpha-1}$};
\node[rectangle,draw] (a12) at (3.6,-1) {$p^{\alpha-2}$};
\node (a21) at (2.2,-2) {$\vdots$};
\node (a22) at (3.6,-2) {$\vdots$};
\node[rectangle,draw] (a31) at (2.2,-3) {$p^4$};
\node (a32) at (3.6,-3) {$p^3$};
\node (a41) at (2.2,-4) {$p^2$};
\node[rectangle,draw] (a42) at (3.6,-4) {$p$};
\node[rectangle,draw] (a51) at (2.2,-5) {$1$};
\node (legende2) at (2.8,-5.5) {$\alpha\equiv 3\Mod 4$};

\draw (a0) -- (a11);
\draw (a0) -- (a12);
\draw (a11) -- (a21);
\draw (a21) -- (a31);
\draw (a31) -- (a41);
\draw (a41) -- (a51);
\draw (a12) -- (a22);
\draw (a22) -- (a32);
\draw (a32) -- (a42);

\node (c0) at (5.6,0) {$p^\alpha$};
\node[rectangle,draw] (c12) at (6.2,-1) {$p^{\alpha-1}$};
\node[rectangle,draw] (c11) at (5,-1) {$p^{\alpha-2}$};
\node (c21) at (5,-2) {$\vdots$};
\node (c22) at (6.2,-2) {$\vdots$};
\node (c31) at (5,-3) {$p^2$};
\node (c32) at (6.2,-3) {$p^3$};
\node[rectangle,draw] (c41) at (5,-4) {$1$};
\node[rectangle,draw] (c42) at (6.2,-4) {$p$};
\node (legende3) at (5.6,-5.5) {$\alpha\equiv 2\Mod 4$};

\draw (c0) -- (c11);
\draw (c0) -- (c12);
\draw (c11) -- (c21);
\draw (c21) -- (c31);
\draw (c31) -- (c41);
\draw (c12) -- (c22);
\draw (c22) -- (c32);
\draw (c32) -- (c42);

\node (d0) at (8.4,0) {$p^\alpha$};
\node (d12) at (9,-1) {$p^{\alpha-1}$};
\node (d11) at (7.8,-1) {$p^{\alpha-2}$};
\node (d21) at (7.8,-2) {$\vdots$};
\node (d22) at (9,-2) {$\vdots$};
\node (d31) at (7.8,-3) {$p^2$};
\node (d32) at (9,-3) {$p^3$};
\node[rectangle,draw] (d41) at (7.8,-4) {$1$};
\node[rectangle,draw] (d42) at (9,-4) {$p$};
\node (legende4) at (8.4,-5.5) {$\alpha\equiv 0\Mod 4$}; 

\draw (d0) -- (d11);
\draw (d0) -- (d12);
\draw (d11) -- (d21);
\draw (d21) -- (d31);
\draw (d31) -- (d41);
\draw (d12) -- (d22);
\draw (d22) -- (d32);
\draw (d32) -- (d42);

\end{tikzpicture}
\end{center}
And we compute $R_k(n)$ again thanks to Proposition \ref{compute}.

\item If $k=up$ and $n=p^\alpha q^\beta$ with $\gcd(u,p)=gcd(u,q)=1$, we get a forest of $\beta+1$ rooted trees $(T_j)_{j=0\cdots \beta}$ with $T_j$ :

\begin{center}
\begin{tikzpicture}
\node (b0) at (0,0) {$p^\alpha q^j$};
\node (b1) at (0,-1) {$p^{\alpha-1}q^j$};
\node (b2) at (0,-2) {$\vdots$};
\node (b3) at (0,-3) {$p^2q^j$};
\node (b4) at (0,-4) {$pq^j$};
\node (b5) at (0,-5) {$q^j$};

\draw (b0) -- (b1);
\draw (b1) -- (b2);
\draw (b2) -- (b3);
\draw (b3) -- (b4);
\draw (b4) -- (b5);
\end{tikzpicture}
\end{center}
Then, the algorithm gives, as in the first case, the size of an optimal $k$-free set in $T_j$:

$$\sum_{i=0}^{\left\lfloor \frac{\alpha-1}{2}\right\rfloor}\varphi(p^{\alpha-2i}q^{\beta-j}).$$
We just need to add the contribution of all $T_j$'s to get the result.

For the case $k=up^2$, this time, it is a consequence of the second case.

\end{enumerate}

\end{proof}

\section{Proof of Theorem \ref{th2}}\label{sec7}

Now, we want to study $k$-free sets in $\llbracket1,n\rrbracket$, and a good way is to consider the partition

$$\llbracket1,n\rrbracket=\bigsqcup_{i\not\equiv0\Mod k}\left(\mathcal{O}^k(i)\bigcap\llbracket1,n\rrbracket\right).$$
Indeed, to be a $k$-free set is equivalent to not have consecutive elements in such orbits (we abusively call orbit of $i$ the set $\mathcal{O}^k(i)\bigcap\llbracket1,n\rrbracket$). Let see now what being maximal for inclusion means in term of orbits. Actually, we can clearly assume that $A$ is a maximal $k$-free set (for inclusion) if and only if for each orbits of $i\not\equiv0\Mod k$, exactly one of the two first elements is in $A$, exactly one of the two last elements is in $A$, there is not consecutive elements, and for all three consecutives elements, there is at least one which is in $A$. That leads us to study the following combinatorial problem :

A set $E\subset\llbracket1,l\rrbracket$ satisfies $(\mathcal{P})$ if :
\begin{itemize}
\item $1 \in E \text{ or } 2\in E$.
\item $l-1\in E \text{ or } l\in E$.
\item $i\in E \Rightarrow (i-1)\notin E \text{ and } (i+1) \notin E$.
\item $\forall i \in \llbracket 2,l-1\rrbracket, \left\{i-1,i,i+1\right\}\cap E\neq \emptyset$.
\end{itemize}
We denote by $h(l)$ the minimal size of a set which satisfies $(\mathcal{P})$ in $\llbracket1,l\rrbracket$.

\begin{lem}\label{comb}
$$h(l)=\left\lceil \frac{l}{3} \right\rceil.$$
\end{lem}

\begin{proof}

First case : $l=3u$. $B=\left\{2, 5, \cdots, 2+3(u-1)\right\}$ satisfies $(\mathcal{P})$ and has a size $u=l/3$. Since we have to take one element among $\left\{3i+1,3i+2,3i+3\right\}, \forall i \in \llbracket0,u-1\rrbracket$, $h(l)\geqslant u$. Then, $h(3u)=u$.

Second case : $l=3u-1$. We consider the following partition : $$\llbracket1,3u-1\rrbracket=\left\{1,2\right\}\bigcup\left(\bigcup_{i\in\llbracket1,u-1\rrbracket}\left\{3i,3i+1,3i+2\right\}\right).$$
Since we must have at least one element from each subsets , we have $h(3u-1)\geqslant u$. But $B=\left\{2, 5, \cdots, 2+3(u-1)\right\}$ has still the good size. Then $h(3u-1)=u$.

Third case : $l=3u-2$, We consider the following partition :
$$\llbracket1,3u-2\rrbracket=\left\{1,2\right\}\bigcup\left(\bigcup_{i\in\llbracket1,u-2\rrbracket}\left\{3i,3i+1,3i+2\right\}\right)\bigcup\left\{3u-3,3u-2\right\}.$$
Since we must have at least one element from each subsets, we have $h(3u-2)\geqslant u$. But $B=\left\{1, 4, \cdots, 1+3(u-1)\right\}$ satisfies $(\mathcal{P})$. Then $h(3u-2)=u$.

\end{proof}

We are now able to prove the Theorem \ref{th2}.

\begin{proof}

If we denote $A_i:=\rrbracket\frac{n}{k^{i+1}},\frac{n}{k^{i}}\rrbracket$, we have $$\llbracket1,n\rrbracket=\bigcup_{i=1}^d A_i$$ where $d=\left[\log_k(n)\right]$. Moreover, $$\left|A_i\right|=\frac{n}{k^{i}}-\frac{n}{k^{i+1}}+\alpha(i)$$ with $\left|\alpha(i)\right|\leqslant1$. And the numbers of $j\not\equiv0\Mod k$ in $A_i$ is $$\left(1-\frac{1}{k}\right)\left(\frac{n}{k^{i}}-\frac{n}{k^{i+1}}+\alpha(i)\right)+\epsilon(i)$$ with $\left|\epsilon(i)\right|\leqslant 1$.
Each element in $A_i$ has an orbit of size $i+1$, then we deduce from the Lemma \ref{comb}:

$$\begin{aligned}\tilde{R_k}(n)&=\sum_{i=0}^d\left\lceil \frac{i+1}{3}\right\rceil\left(\left(1-\frac{1}{k}\right)\left(\frac{n}{k^{i}}-\frac{n}{k^{i+1}}+\alpha(i)\right)+\epsilon(i)\right)\\
&=\sum_{i=0}^d\left\lceil\frac{i+1}{3}\right\rceil\left(1-\frac{1}{k}\right)\left(\frac{n}{k^{i}}-\frac{n}{k^{i+1}}\right)+O(\log_k^{2}(n)).
\end{aligned}$$
To study this sum, we group together by three the terms with same integer part, in order to get a telescopic behaviour. Thus, we get :

$$\begin{aligned}\tilde{R_k}(n)&=\left(1-\frac{1}{k}\right)\sum_{i=0}^{\left\lfloor \frac{d}{3}\right\rfloor}\left(i+1\right)\left(\frac{n}{k^{3i}}-\frac{n}{k^{3i+1}}+\frac{n}{k^{3i+1}}-\frac{n}{k^{3i+2}}+\frac{n}{k^{3i+2}}-\frac{n}{k^{3i+3}}\right)\\
&+\beta(n)+O(\log_k^{2}(n))\\
&=\left(1-\frac{1}{k}\right)\sum_{i=0}^{\left\lfloor \frac{d}{3}\right\rfloor}\left(i+1\right)\left(\frac{n}{k^{3i}}-\frac{n}{k^{3i+3}}\right)+\beta(n)+O(\log_k^{2}(n))
\end{aligned}
$$
with

$$\beta(n)\leqslant\left(1-\frac{1}{k}\right)\times2\left(\frac{d}{3}+1\right)\left(\frac{n}{k^{d-1}}-\frac{n}{k^{d+1}}\right)=O(\log_k(n)).$$

Finally, $$\begin{aligned}\tilde{R_k}(n)&=\left(1-\frac{1}{k}\right)\sum_{i=0}^{\left\lfloor \frac{d}{3}\right\rfloor}\frac{n}{k^{3i}}+O(\log_k^2(n))\\
&=\frac{k^2}{k^2+k+1}n+O(\log_k^2(n)).
\end{aligned}$$

\end{proof}

\section*{Acknowledgements}

This work is part of the PhD of the author. He wishes to thank his advisor Alain Plagne for his support during the preparation of this paper. We also thank Jeremy Le Borgne for his help and careful reading. The author was supported  by an ANR grant Cæsar, number ANR 12 - BS01 - 0011.

\end{document}